\documentclass[letterpaper,12pt,reqno]{amsart}
\usepackage{graphicx,amsmath,amsthm,amssymb}
\usepackage{epsf, subfigure, verbatim}

\usepackage{srcltx}
\usepackage[latin1]{inputenc}
\usepackage[T1]{fontenc}

\usepackage{fullpage}

\usepackage{amsthm}

\theoremstyle{plain}
\newtheorem{X}{X}[section]
\newtheorem{corollary}[X]{Corollary}
\newtheorem{definition}[X]{Definition}

\newtheorem{lemma}[X]{Lemma}
\newtheorem{proposition}[X]{Proposition}
\newtheorem{theorem}[X]{Theorem}

\theoremstyle{definition}
\newtheorem{remark}[X]{Remark}

\title{Residue classes containing an unexpected number of primes}
\author{Daniel Fiorilli}
\address{D\'epartement de math\'ematiques et de statistique \\ Universit\'e de Montr\'eal \\ CP 6128, succ.\ Centre-ville \\ Montr\'eal, QC \\ Canada H3C 3J7}
\email{fiorilli@dms.umontreal.ca}

\begin{document}

\begin{abstract}
\noindent We fix a non-zero integer $a$ and consider arithmetic progressions $a \bmod q$, with $q$ varying over a given range. We show that for certain specific values of $a$, the arithmetic progressions $a\bmod q$ contain, on average, significantly fewer primes than expected. 
\end{abstract}

\maketitle

\section{Introduction}

 The prime number theorem for arithmetic progressions asserts that $$\psi(x;q,a)\sim \psi(x)/\phi(q)$$ for any $a$ and $q$ such that $(a,q)=1$. Another way to say this is that the primes are equidistributed in the $\phi(q)$ arithmetic progressions $a \bmod q$ with $(a,q)=1$. 

 Fix an integer $a\neq 0$. We will be interested in the number of primes in the arithmetic progressions $a \bmod q$ with $q$ varying in certain ranges, and we will show that for specific values of $a$, there are significantly fewer primes in these arithmetic progressions than in typical arithmetic progressions. Consider the average value of $\psi(x;q,a)- \psi(x)/\phi(q)$ over $q$. One might expect that no matter what the value of $a$ is, the cancellations in these oscillating terms will force the average to be very small. However it turns out that the average is highly dependent on the arithmetical properties of $a$.

Here is the main result of the paper.

\begin{theorem}
\label{premier théorème}
Fix an integer $a\neq 0$ and let $M=M(x) \leq (\log x)^B$  where $B>0$ is a fixed real number. The average error term in the usual approximation for the number of primes $p\equiv a \bmod q$ with $p\leq x$, $p\neq a$, where $(q,a)=1$ and $q\leq x/M$, is

\begin{equation}
\label{enonce premier thm}
\frac1{\frac{\phi(a)}{a}\frac{x}{M}  }\sum_{\substack{q\leq \frac{x} {M}  \\(q,a)=1}} \left( \psi(x;q,a)-\Lambda(a)-\frac{\psi(x)}{\phi(q)}\right) = \mu(a,M)+O_{a,\epsilon,B} \left(\frac{1}{M^{\frac {205}{538}-\epsilon}}\right)
\end{equation}
with $$ \mu(a,M):=  \begin{cases}
                     0 &\text{ if } \omega(a)\geq 2 \\
			-\frac 12 \log p &\text{ if } a=\pm p^e \\
			-\frac 12 \log M - C_5 &\text{ if } a=\pm 1,
                    \end{cases}
$$
where
$$ C_5:= \frac 12 \left(\log 2\pi + \gamma +\sum_p \frac {\log p}{p(p-1)}+1 \right).$$

\end{theorem}

\begin{remark}
 Assuming Lindelöf's hypothesis, we can replace the error term in (\ref{enonce premier thm}) by 
 $O_{a,\epsilon,B}\left( \frac{1}{M^{1/2-\epsilon}}\right).$
\end{remark}

\begin{remark}
We subtracted $\Lambda(a)$ from $\psi(x;q,a)$ in (\ref{enonce premier thm}) because the arithmetic progression $a \bmod q$ contains the prime power $p^e$ for all $q$ if $a=p^e$. 
\end{remark}

\begin{remark}
\label{remarque friedlander}
It may be preferable to replace $\psi(x)$ by $\psi(x,\chi_0)$ in Theorem \ref{premier théorème}, since the quantity $$\psi(x;q,a)- \psi(x,\chi_0)/\phi(q)$$ is the discrepancy (with signs) of the sequence of primes in the reduced residue classes mod $q$. One can do this with a negligible error term.
\end{remark}

\section*{Acknowledgements}

I would like to thank my supervisor Andrew Granville for many helpful comments and his advice in general, as well as my colleagues Farzad Aryan, Mohammad Bardestani, Tristan Freiberg and Kevin Henriot for many fruitful conversations. I would also like to thank John Friedlander for suggesting Remark \ref{remarque friedlander} and for other helpful comments. L'auteur est titulaire d'une bourse doctorale du Conseil de recherches en sciences naturelles et en g\'enie du Canada.

\section{Past Results}

The study of the discrepancy $\psi(x;q,a)-x/\phi(q)$ on average has been a fruitful subject over the past decades. For example, the celebrated theorem of Bombieri-Vinogradov gives a bound on the sum of the mean absolute value of the maximum of this discrepancy over all $1\leq a <q$ with $(a,q)=1$, summed over $q\leq x^{1/2-o(1)}$. The Hooley-Montgomery refinement of the Barban-Davenport-Halberstam Theorem gives an estimation of the variance of $\psi(x;q,a)-x/\phi(q)$, again for all values of $a$ in the range $1\leq a <q$ with $(a,q)=1$ for $q<x/(\log x)^A$. The mean value of $\psi(x;q,a)-x/\phi(q)$ was studied for fixed values of $a$ for $q\geq x^{1/2}$ (see \cite{BFI},\cite{fouvry} or \cite{FG}), and bounds on this mean value turned out to be applicable to Titchmarsh's divisor problem, first solved by Linnik. The best result so far for this problem was obtained by Friedlander and Granville.

\begin{theorem}[Friedlander, Granville]
\label{FG}
Let $0<\lambda<1/4$, $A>0$ be given. Then uniformly for $0<|a|<x^{\lambda}$, $2\leq Q \leq x/3$ we have 
\begin{equation}
\sum_{\substack{Q<q\leq 2Q \\ (q,a)=1}} \left( \psi(x;q,a)-\frac x{\phi(q)} \right) \ll_{\lambda,A} 2^{\omega(a)} Q \log(x/Q) + \frac x{(\log x)^A} +Q\log |a|.
\end{equation}
\end{theorem}
\begin{remark}
 If $a$ is not a prime power the term $Q\log |a|$ may be deleted.
\end{remark}
Theorem \ref{FG} is a refinement of the deep results of Bombieri-Friedlander-Iwaniec \cite{BFI} and of Fouvry \cite{fouvry}, and makes use of the dispersion method combined with Fourier analysis and involved estimates on Kloosterman sums. 

The main method used in our paper, which we will refer to as the "divisor switching" technique, stemmed from the work of Dirichlet on the divisor problem. Variants of his "hyperbola method" were subsequently used in many different contexts, and have become a very important tool in analytic number theory. The variant which will be used in this paper is very similar to that of Hooley \cite{hooley}.
 

%
%
%
%
%

\section{Main results}

What we will actually prove in Section \ref{section further results} is a uniform version of Theorem \ref{premier théorème} (which is too technical to state at this point), and we will see that the proportion of integers $|a|\leq x^{\frac 14-\epsilon}$ for which we get an error term of $O(M^{-\eta})$ in \eqref{enonce premier thm} is at least $(1-4\eta)e^{-\gamma}$ (here and throughout, $\gamma$ denotes the Euler-Mascheroni constant). 

A number of constants will appear in the paper. 

\begin{definition}
We define
 $$C_1(a):=\frac{ \zeta(2)\zeta(3)}{\zeta(6)}\frac {\phi(a)}{a}\prod_{p\mid a} \left( 1-\frac 1 {p^2-p+1}\right),$$
$$C_3(a):=C_1(a)\left(\gamma -1 - \sum_{p} \frac {\log p} {p^2-p+1}+\sum_{p\mid a} \frac{p^2\log p}{(p-1)(p^2-p+1)}\right), $$
$$ C_5:= \frac 12 \left(\log 2\pi + \gamma +\sum_p \frac {\log p}{p(p-1)}+1 \right).$$
 
\end{definition}

We will denote by $\omega(a)$ the number of distinct prime factors of $a$. Note that provided $0\neq|a|\leq x$ and $\sum_{d\mid a}1=o\left( \frac xM \frac{\phi(a)}a\right)$, there are $\sim \frac{ \phi(a)}{a} \frac xM$ terms in the sum over $1\leq q \leq x/M$ with the condition $(q,a)=1$.




The results of Theorem \ref{premier théorème} become worse as $M$ gets smaller. The reason for this is that when $M=O(1)$ (that is the relevant arithmetic progressions have bounded length), the estimates become dramatically different. This case is handled in the next Proposition.


\begin{proposition}
\label{M=1}
Fix $A$ and $\lambda<\frac 14$, two positive real numbers. For $a$ in the range $0<|a|\leq x^{\lambda}$ such that $\omega(a)\leq 10\log \log x$ we have 
\begin{equation}
\frac 1{\frac{\phi(a)}{a}x}\sum_{\substack{q\leq  x   \\(q,a)=1}} \left( \psi(x;q,a)-\Lambda(a)-\frac{\psi(x)}{\phi(q)} \right) = \frac a{\phi(a)}C_3(a)+O_{A,\lambda}\left( \frac 1 {(\log x)^A} \right).
\end{equation}
More generally, for $M\geq 1$ a fixed integer we have 
 \begin{equation}
\frac{1}{\frac{\phi(a)}{a} \frac xM}\sum_{\substack{q\leq \frac x M  \\(q,a)=1}} \left( \psi(x;q,a)-\Lambda(a)-\frac{\psi(x)}{\phi(q)} \right) =\mu'(a,M) +O_{A,\lambda}\left( \frac {1} {(\log x)^A} \right),
\end{equation}
where $$\mu'(a,M):= \frac a{\phi(a)}M\Bigg(C_1(a)\log M+C_3(a)-\sum_{\substack{r\leq M\\(r,a)=1}} \frac 1 {\phi(r)} \left( 1- \frac rM\right)\Bigg).$$
\end{proposition}
Note that $\mu'(a,1)=\frac{a}{\phi(a)}C_3(a)$. 

\begin{remark}
It is possible to give such an estimate for any $M\in \mathbb R_{\geq 1}$, however the formula is simpler when $M$ is an integer, so we prefered to just include the latter case.
\end{remark}

By inverting the order of summation in Proposition \ref{M=1}, one gets the following corollary, which is an example of application of the results of Bombieri, Fouvry, Friedlander, Granville and Iwaniec (see \cite{BFI}, \cite{fouvry} and \cite{FG}).
\begin{corollary}[Titchmarsh's divisor problem]
\label{Titchmarsh}
Fix $A$ and $\lambda<\frac 14$, two positive real numbers.  For $a$ in the range $0<|a|\leq x^{\lambda}$ such that $\omega(a)\leq 10\log \log x$ we have 
\begin{equation}
 \sum_{|a|<n\leq x} \Lambda(n)\tau(n-a) = C_1(a) x\log x + ( 2 C_3(a)+C_1(a))x + O_{A,\lambda}\left( \frac x {(\log x)^A} \right).
\end{equation}

\end{corollary}
Note that the constant $10$ in Proposition \ref{M=1} and Corollary \ref{Titchmarsh} can be replaced by an arbitrarily large real number.

\section{Notation}

\begin{definition}
For $n\neq 0$ an integer (possibly negative), we define

 \begin{equation*}
  \Lambda(n):=\begin{cases}
               \log p &\text{ if } n=p^e \\
		0 &\text{ otherwise, }
              \end{cases}
\hspace{2cm}
  \vartheta(n):=\begin{cases}
               \log p &\text{ if } n=p \\
		0 &\text{ otherwise. }
              \end{cases}
 \end{equation*}
\end{definition}

\begin{definition}
\begin{equation*}
 \psi(x;q,a) := \sum_{\substack{ n \leq x\\ n \equiv a\bmod q}}  \Lambda(n), \hspace{2cm} \theta(x;q,a) := \sum_{\substack{ n \leq x \\ n\equiv a\bmod q}}  \vartheta(n).
\end{equation*}
\end{definition}

The following definition is non-standard but will be useful in the proofs.
\begin{definition}
\begin{equation}
 \psi^*(x;q,a) := \sum_{\substack{|a|< n \leq x \\ n\equiv a\bmod q}}   \Lambda(n),
\end{equation}

\begin{equation}
 \theta^*(x;q,a) := \sum_{\substack{ |a|< n \leq x \\ n\equiv a\bmod q}}  \vartheta(n).
\end{equation}
\end{definition}

We will need to consider the prime divisors of $a$ which are less than $M$.

\begin{definition}
\label{a_M}
For $a$ an integer and $M>0$ a real number, we define
\begin{equation}
a_M:= \prod_{\substack{p\mid a \\ p\leq M}} p. 
\end{equation}
\end{definition}

The error term $E(M,a)$ will be defined depending on the context, so one has to pay attention to its definition in every statement.

\section{Lemmas}

We begin by recalling the divisor switching technique.
\begin{lemma}
 \label{interchanger les diviseurs}
Let $a$ be an integer such that $0<|a|\leq x$ and let $M=M(x)$ such that $1\leq M < x$. We have
\begin{equation} \sum_{\substack{\frac x{M} < q\leq x  \\(q,a)=1}} \sum_{\substack{ |a|< p\leq x \\ p\equiv a \bmod q  }} \log p = \sum_{\substack{1\leq r < (x - a)\frac{M}{x}   \\(r,a)=1}} \sum_{\substack{ r\frac x{M} +a< p\leq x \\ p\equiv a \bmod r  \\ p> |a|}} \log p +O(|a|\log x).
\end{equation}
\end{lemma}
\begin{proof}
Clearly,

\begin{equation}
 \sum_{\substack{\frac x{M} < q\leq x  \\(q,a)=1}} \sum_{\substack{ |a|< p\leq x \\ p\equiv a \bmod q  }} \log p  =  \sum_{\substack{\frac x{M} < q\leq x  \\(q,a)=1}} \sum_{\substack{ |a|< p\leq x \\ p\equiv a \bmod q \\ p> a+\frac x{M} }} \log p.
\label{switch}
\end{equation}

We now apply Hooley's variant of the divisor switching technique (see \cite{hooley}). Setting $p-a =rq $ in  (\ref{switch}), one can sum over $r$ instead of summing over $q$. Now $(r,a)=1$, else $(p,a)>1$ so $p \mid a $, but this is impossible since $p>|a|$. Taking $a>0$ for now, we get that \eqref{switch} is equal to

 \begin{equation}
\label{switch 2}
\sum_{\substack{\frac x{M} < q\leq x-a  \\(q,a)=1}} \sum_{\substack{ |a|< p\leq x \\ p\equiv a \bmod q \\ p> a+\frac x{M} }} \log p =\sum_{\substack{1\leq r < (x - a)\frac{M}{x}   \\(r,a)=1}} \sum_{\substack{ r\frac x{M} +a< p\leq x \\ p\equiv a \bmod r  \\ p> |a|}} \log p.
\end{equation}
If we had $a<0$, additional terms would be needed in passing from the right hand side of \eqref{switch} to the left hand side of \eqref{switch 2}. These additional terms are
$$ \ll \sum_{x<q\leq x-a} \log x \ll |a|\log x.$$ 
The proof is complete.

\end{proof}

\begin{lemma}

We have the following estimates:
 \label{somme inverse euler}
\begin{equation}
\label{equation somme inverse euler}
 \sum_{\substack{n\leq M\\ (n,a)=1}} \frac1{\phi(n)} =  C_1(a) \log M + C_1(a) + C_3(a) + O\left( 2^{\omega(a)}\frac{\log M}{M}\right),
\end{equation}

\begin{equation}
\label{equation somme inverse euler avec n}
 \sum_{\substack{n\leq M\\ (n,a)=1}} \frac n{\phi(n)} =  C_1(a) M + O\left( 2^{\omega(a)} \log M \right).
\end{equation}
\end{lemma}
Note that without loss of generality, we can replace $a$ by $a_M$ on the right side of \eqref{equation somme inverse euler} and \eqref{equation somme inverse euler avec n}.
\begin{proof}
The proof of \eqref{equation somme inverse euler} is very similar that of Lemma 13.1 in \cite{FGHM}. One first has to prove the following estimate: 
\begin{equation}
\label{estimation preliminaire}
  \sum_{\substack{n\leq M \\(n,a)=1}} \frac 1 n = \frac{\phi(a)}{a} \left(\log M + \gamma + \sum_{p\mid a} \frac{\log p}{p-1} \right) + O\left( \frac{2^{\omega(a)}}{M} \right).
 \end{equation}
One then writes
\begin{equation}
\sum_{\substack{n\leq M\\ (n,a)=1}} \frac1{\phi(n)} = \sum_{\substack{n\leq M\\ (n,a)=1}}  \frac 1 n \sum_{d\mid n} \frac{\mu^2(d)}{\phi(d)} = \sum_{\substack{d\leq M\\ (d,a)=1}} \frac{\mu^2(d)}{d \phi(d)} \sum_{\substack{r\leq M/d \\ (r,a)=1} } \frac 1r
\label{debut du truc}
\end{equation}
and inserts the estimate (\ref{estimation preliminaire}) into (\ref{debut du truc}). The final step is to bound the tail of the sums and to compute the following constants: $$\sum_{\substack{(d,a)=1}} \frac{\mu^2(d)}{d \phi(d)} =\frac{ \zeta(2)\zeta(3)}{\zeta(6)}\prod_{p\mid a} \left( 1-\frac 1 {p^2-p+1}\right),$$ $$\sum_{\substack{(d,a)=1}} \frac{\mu^2(d)}{d \phi(d)} \log d = \frac{ \zeta(2)\zeta(3)}{\zeta(6)}\prod_{p\mid a} \left( 1-\frac 1 {p^2-p+1}\right)  \sum_{p\nmid a} \frac{\log p}{p^2-p+1} . $$
The proof of \eqref{equation somme inverse euler avec n} goes along the same lines.
\end{proof}

The delicacy of the analysis forces us to give some details about the "trivial" estimates for the prime counting functions. 

\begin{lemma}
\label{psi et theta}
For any real number $\epsilon>0$ and integer $a$ with $0<|a|\leq x$,
\begin{equation}
 \sum_{\substack{q\leq x\\(q,a)=1}} (\psi^*(x;q,a)-\theta^*(x;q,a)) \ll_{\epsilon} x^{1/2+\epsilon}.
\end{equation}

\end{lemma}
\begin{proof}

\begin{align*}
 & \sum_{\substack{q\leq x\\(q,a)=1}} (\psi^*(x;q,a)-\theta^*(x;q,a)) \\ & \hspace{60pt} \leq \sum_{\substack{q\leq x}} \sum_{\substack{|a|<p^e\leq x \\ p^e\equiv a \bmod q \\ e\geq 2}} \log p
\leq \sum_{2\leq e \leq \frac{\log x}{\log 2}} \sum_{p\leq x^{1/e}} \sum_{\substack{q \leq x \\ q \mid p^e-a}} \log p 
 \\ &\hspace{60pt} \leq \log x \sum_{2\leq e \leq \frac{\log x}{\log 2}} \sum_{p\leq x^{1/e}} \tau(p^e-a) \ll_{\epsilon} x^{\epsilon/2} \sum_{2\leq e \leq \frac{\log x}{\log 2}} \pi(x^{1/e}) \\ & \hspace{60pt} \ll_{\epsilon} x^{1/2+\epsilon}.
\end{align*}

\end{proof}

\begin{lemma}
\label{enlever *}
 Let $a$ be an integer with $0\neq|a|\leq x$ and let $1\leq  Q \leq x$. We have

\begin{equation}
 \sum_{\substack{q\leq Q\\(q,a)=1}} (\psi(x;q,a)-\Lambda(a)-\psi^*(x;q,a)) = O(|a| (\log |a|)^2),
\end{equation}

\begin{equation}
 \sum_{\substack{q\leq Q\\(q,a)=1}} (\theta(x;q,a)-\vartheta(a)-\theta^*(x;q,a)) = O(|a| (\log |a|)^2).
\end{equation}

\end{lemma}
\begin{proof}
Note that as soon as $q>2|a|$, there are no integers congruent to $a \bmod q$ in the interval $[1,|a|)$. We then have
\begin{multline*}
 \sum_{\substack{q\leq Q\\(q,a)=1}} (\psi(x;q,a)-\Lambda(a)-\psi^*(x;q,a)) = \sum_{\substack{q\leq Q\\(q,a)=1}} \sum_{\substack{1\leq n < |a|\\n\equiv a \bmod q }} \Lambda(n) \\ \leq \sum_{\substack{q\leq 2|a|\\(q,a)=1}} \log|a| \sum_{\substack{1\leq n < |a| \\ n \equiv a \bmod q}}1  \ll \sum_{\substack{q\leq 2|a|\\ (q,a)=1}} \log |a| \left( \frac{|a|}q \right) \ll |a| (\log |a|)^2.
\end{multline*}

The proof for $\theta$ and $\theta^*$ is similar.
\end{proof}

\begin{lemma}
\label{psi et x}
Let $I\subset [1,x]\cap \mathbb N$. We have
 \begin{equation}
\sum_{q\in I} \left( \frac x{\phi(q)} - \frac{\psi(x)}{\phi(q)} \right) \ll xe^{-C\sqrt{\log x}},
 \end{equation}
where $C$ is an absolute positive constant.
\end{lemma}
\begin{proof}
 This follows from Lemma \ref{somme inverse euler} and the prime number theorem.
\end{proof}

To prove Lemma \ref{somme inverse euler avec poids} we will need bounds on $\zeta(s)$.

\begin{lemma}
\label{bornes sur zeta}
 Define $\theta:=\frac{32}{205}$ and take any $\epsilon>0$. In the region $|\sigma+it-1|>\frac 1 {10}$, we have $$ \zeta(\sigma+it)\ll_{\epsilon} (|t|+1)^{\mu(\sigma)+\epsilon}, $$ where
\begin{equation*}
\label{def de mu}
\mu(\sigma)=\begin{cases}
             1/2-\sigma &\text{ if } \sigma \leq 0 \\
1/2 +(2\theta-1)\sigma &\text{ if } 0 \leq \sigma \leq 1/2 \\
2\theta(1-\sigma) &\text{ if } 1/2\leq  \sigma \leq 1 \\
0  &\text{ if } \sigma \geq 1. \\
            \end{cases}
\end{equation*}
\end{lemma}
\begin{proof}
 For the values outside the critical strip, see for example section II.3.4 of \cite{tenenbaum}. In the critical strip, we use an estimate due to Huxley \cite{huxley}, which showed that $\zeta(1/2+it)\ll_{\epsilon} (|t|+1)^{\frac{32}{205}+\epsilon}$. The lemma then follows by convexity of $\mu$.

\end{proof}

\begin{remark}
\label{remarque lindelof}
 Under Lindelöf's hypothesis, the conclusion of Lemma \ref{bornes sur zeta} holds with $\theta=0$.
\end{remark}

\begin{lemma}[Perron's formula]
\label{formule de Perron}
Let $0<\kappa<1$, $y>0$ and define 
\begin{equation*}
 h(y):=\begin{cases} 
0 &\text{ if } 0<y<1 \\
1-\frac 1y & \text{ if } y\geq 1. 
       \end{cases}
\end{equation*}
We have 
\begin{equation*}
 h(y) = \frac 1 {2\pi i} \int_{(\kappa)} \frac{y^s}{s(s+1)} ds.
\end{equation*}
Moreover, for $T\geq 1$ and positive $y\neq 1$, we have the estimate
\begin{equation*}
 h(y) = \frac 1 {2\pi i} \int_{\kappa-iT}^{\kappa+iT} \frac{y^s}{s(s+1)} ds + O\left( \frac{y^{\kappa}}{T^2 |\log y|} \right).
\end{equation*}
Finally, for $y=1$,
\begin{equation*}
 0=h(1)=\frac 1 {2\pi i} \int_{\kappa-iT}^{\kappa+iT} \frac{ds}{s(s+1)} + O \left( \frac 1T \right).
\end{equation*}

\end{lemma}

\begin{proof}
 The first assertion is an easy application of the residue theorem. 

Now take $y > 1$. We have again by the residue theorem that for any large integer $K\geq 3$ and for $T\geq 1$,

\begin{align*}
 \frac 1 {2\pi i} \int_{\kappa-iT}^{\kappa+iT} \frac{y^s}{s(s+1)} ds-h(y) = \frac 1 {2\pi i}\left( \int_{\kappa-iT}^{\kappa-K-iT} +\int_{\kappa-K-iT}^{\kappa-K+iT} +\int_{\kappa-K+iT}^{\kappa+iT}\right) \frac{y^s}{s(s+1)} ds 
\\ \ll \frac 1{T^2} \int_{\kappa-K}^{\kappa} y^{\sigma} d\sigma +  \frac{y^{\kappa-K}}{|\kappa-K|^2} \int_{\kappa-K-iT}^{\kappa-K+iT} |ds| \ll \frac{y^{\kappa}}{T^2|\log y|}+T\frac{y^{\kappa-K}}{|\kappa-K|^2}.
\end{align*}
We deduce the second assertion of the lemma by letting $K$ tend to infinity. The proof is similar in the case $0<y<1$.

The last case remaining is for $y=1$. We have

\begin{align*}
 \frac 1 {2\pi i} \int_{\kappa-iT}^{\kappa+iT} \frac{ds}{s(s+1)} &= \frac 1 {2\pi i} \log \left(\frac{1+\frac 1{\kappa-iT}}{1+\frac 1{\kappa+iT}} \right) \\ &= \frac 1 {2\pi i} \log \left(1+O\left( \frac 1T\right) \right) \\ &= O\left( \frac 1T\right),
\end{align*}
which concludes the proof.

\end{proof}

The following is a crucial lemma estimating a weighted sum of the reciprocal of the totient function.

\begin{lemma}
\label{somme inverse euler avec poids}

Let $a\neq 0$ be an integer and $M\geq 1$ be a real number.

If $\omega(a_M)\geq 1$,
\begin{equation}
\label{enonce premier} 
\sum_{\substack{n\leq M\\ (n,a)=1}} \frac1{\phi(n)}\left( 1-\frac nM\right) =  C_1(a_M) \log M +  C_3(a_M) +\frac{\phi(a_M)}{a_M}\frac {\Lambda(a_M)}{2M} +E(M,a).
\end{equation}

If $a_M= 1$,
\begin{equation}
\label{enonce 1}
 \sum_{\substack{n\leq M\\ (n,a)=1}} \frac1{\phi(n)}\left( 1-\frac nM\right) =  C_1(1) \log M +  C_3(1) +\frac 12 \frac{\log M}{M}  + \frac{C_5}{M}+E(M,a).
\end{equation}

There exists $\delta>0$ such that the error term $E(M,a)$ satisfies

\begin{equation}
E(M,a) \ll_{\epsilon} \frac{\prod_{p\mid a_M} \left( 1+\frac 1{p^{\delta}}\right)} M \left(\frac{a_M} M \right)^{\frac{205}{538}-\epsilon}.
\end{equation}
\end{lemma}

\begin{remark}
 
Under Lindelöf's hypothesis, 
\begin{equation}
 E(M,a) \ll_{\epsilon} \frac{\prod_{p\mid a_M} \left( 1+\frac 1{p^{\delta}}\right)} M \left(\frac{a_M} M \right)^{1/2-\epsilon}.
\end{equation}
\end{remark}

\begin{proof}

Note first that we need only to consider the prime factors of $a$ less than or equal to $M$, since for $1\leq n \leq M$, $(n,a)=1 \Leftrightarrow (n,a_M)=1$.
 
To calculate our sum we will write it as a contour integral and shift contours, showing that the contribution of the shifted contours is negligible and obtaining the main terms from the residues at the poles.

Setting $\kappa=\frac 1 {\log M}$ in Lemma \ref{formule de Perron},

\begin{align}
\begin{split}
\label{M entierrr}
 \sum_{\substack{n\leq M\\ (n,a)=1}} &\frac1{\phi(n)} \left( 1-\frac nM\right) = \sum_{\substack{(n,a_M)=1}} \frac1{\phi(n)}h\left( \frac Mn\right) \\
&=\sum_{\substack{(n,a_M)=1}} \frac1{\phi(n)} \frac 1 {2\pi i} \int_{\kappa-iT}^{\kappa+iT}  \left(\frac Mn\right)^s  \frac{ds}{s(s+1)}\\  &\hspace{60pt}+O\left( \frac 1{T^2}\sum_{n\neq M} \frac 1{\phi(n) |\log M/n|}   \left(\frac Mn\right)^{\kappa} +\frac{\log M}{TM}   \right)
\\
&= \frac 1 {2\pi i} \int_{\kappa-iT}^{\kappa+iT} \left(\sum_{(n,a_M)=1}\frac 1 {n^s \phi(n)}  \right) \frac{M^s}{s(s+1)} ds + O_M\left( \frac{1}{T } \right).
\end{split}
\end{align}
In the last step we used the elementary estimates 
$$\sum_{n\leq M} \frac 1 {\phi(n)} \ll \log M  \hspace{0.5cm} \text{and} \hspace{0.5cm} \sum_{n > M} \frac 1 {n^\kappa \phi(n)} \ll \log M.$$ 
Now taking Euler products we compute that
\begin{equation}
\sum_{(n,a_M)=1}\frac 1 {n^s \phi(n)} = \mathfrak{S}_{a_M}(s)\zeta(s+1) \zeta(s+2)Z(s)
\end{equation}
where
\begin{equation}
 \mathfrak{S}_{a_M}(s):=\prod_{\substack{p\mid a\\ p\leq M}} \left( 1-\frac 1 {p^{s+1}} \right)\left( 1+\frac 1 {(p-1)p^{s+1}} \right)^{-1}
\end{equation}
and
\begin{equation}
\label{definition de Z}
 Z(s):=\prod_p \left(    1+\frac{1}{p(p-1)} \left(\frac 1 {p^{s+1}}-\frac 1 {p^{2s+2}}\right)\right),
\end{equation}
which converges for $\Re s>-3/2$. Therefore, \eqref{M entierrr} becomes
\begin{multline}
\label{equation 40}
 \sum_{\substack{n\leq M\\ (n,a)=1}} \frac1{\phi(n)}\left( 1-\frac nM\right) = \frac 1 {2\pi i} \int_{\kappa-iT}^{\kappa+iT} \mathfrak{S}_{a_M}(s)\zeta(s+1)\zeta(s+2) Z(s) \frac{M^s}{s(s+1)} ds \\+ O_{M}\left( \frac{1}{T } \right).
\end{multline}
The different results for different values of $\omega(a_M)$ come from the pole at $s=-1$. We see that $\mathfrak{S}_{a_M}(s)$ has a zero of order $\omega(a_M)$ at $s=-1$ whereas 
$$  \frac{\zeta(s+1)\zeta(s+2)}{s(s+1)} $$
has a pole of order two at $s=-1$. Hence the product has no pole if $\omega(a_M)\geq 2$, a pole of order one if $\omega(a_M)=1$, and a pole of order two if $\omega(a_M)=0$.
We now shift the contour of integration to the left until the line $\Re(s)=\sigma$, where $-1-\frac 1{2+4\theta}<\sigma < -1$ and $\theta:=\frac{32}{205}$. The right hand side of \eqref{equation 40} becomes
\begin{multline}
\label{bouger parcours d'integration}
 =P_T+ \frac 1 {2\pi i} \int_{\sigma-iT}^{\sigma+iT} \mathfrak{S}_{a_M}(s)\zeta(s+1)\zeta(s+2) Z(s) \frac{M^s}{s(s+1)} ds   \\+O_M\left(\frac 1T \right)+ O_a\left( \frac{(\log T)^2}{T^2}\left( T^{1/6}+\frac{T^{1/2}}{M^{1/2}}+\frac{T^{7/6}}{M} \right) \right).
\end{multline}
Here, $P_T$ denotes the sum of all residues in the box $\sigma  \leq \Re s \leq \frac 1 {\log M}$ and $|\Im s| \leq T $. The second error term in \eqref{bouger parcours d'integration} comes from the horizontal integrals which we have bounded using Lemma \ref{bornes sur zeta} (note that $\theta<1/6$). 
Taking $T\rightarrow \infty$ yields

\begin{equation}
\label{integrale}
 \sum_{\substack{n\leq M\\ (n,a)=1}} \frac1{\phi(n)}\left( 1-\frac nM\right)=P_{\infty}+ E(M,a),
\end{equation}
where $$E(M,a):=\frac 1 {2\pi i} \int_{(\sigma)} \mathfrak{S}_{a_M}(s)\zeta(s+1)\zeta(s+2) Z(s) \frac{M^s}{s(s+1)} ds.$$ 
Now on the line $\Re s = \sigma$ we have the bound (note that $0<-1-\sigma<\frac 1{2+4\theta}$)

$$\mathfrak S_{a_M}(s) \ll a_M^{-1-\sigma} \prod_{p\mid a_M} \left(1+\frac 1{p^{\delta}} \right),$$
for some $\delta>0$.
Combining this with Lemma \ref{bornes sur zeta} 
yields
\begin{align}
\begin{split}
\label{borne sur E(M,a) 1}
E(M,a) &\ll_{\sigma}  \prod_{p\mid a_M} \left(1+\frac 1{p^{\delta}} \right) a_M^{-1-\sigma}  \int_{-\infty}^{\infty} |\zeta(\sigma+1+it)||\zeta(\sigma+2+it)| \frac{M^{\sigma}}{(|t|+1)^2} dt \\
&\ll  \frac {\prod_{p\mid a_M} \left(1+\frac 1{p^{\delta}} \right)} M \left(\frac{a_M}M\right)^{-1-\sigma} \int_{-\infty}^{\infty} \frac{(|t|+1)^{1/2-(\sigma+1)} (|t|+1)^{2\theta (1-(\sigma+2))}}{(|t|+1)^2} dt
 \\  &\ll_{\sigma}  \frac {\prod_{p\mid a_M} \left(1+\frac 1{p^{\delta}} \right)}M \left(\frac{a_M}M\right)^{-1-\sigma},
\end{split}
\end{align}
since $1/2-(\sigma+1) + 2\theta (1-(\sigma+2)) <1$ by our choice of $\sigma$. The claimed bound on $E(M,a)$ then follows by taking $\sigma:=-1-\frac 1{2+4\theta}+\epsilon$ in \eqref{borne sur E(M,a) 1}.

It remains to compute $P_{\infty}$ which is the sum of the residues of $\mathfrak{S}_{a_M}(s)\zeta(s+1)\zeta(s+2) Z(s) \frac{M^s}{s(s+1)}$ in the region $\sigma  \leq \Re s \leq  \frac 1 {\log M}$. Note that $\mathfrak{S}_{a_M}(s)$ has poles on the lines $\Re s = -1-\frac{\log(p-1)}{\log p}$, however these poles are cancelled by the zeros of $Z(s)$. Thus the only possible singularities of $\mathfrak{S}_{a_M}(s)\zeta(s+1)\zeta(s+2) Z(s) \frac{M^s}{s(s+1)}$ in the region in question are at the points $s=0$ and $s=-1$.

Now a lengthy but straightfoward computation shows that we have a double pole at $s=0$ with residue equal to 
$ C_1(a_M) \log M +  C_3(a_M)$.

As for $s=-1$, we have to consider three cases.

If $ \omega(a_M)\geq 2$, then $\mathfrak{S}_{a_M}(s) = O((s+1)^2)$ around $s=-1$, so $\mathfrak{S}_{a_M}(s)\zeta(s+1)\zeta(s+2) Z(s) \frac{M^s}{s(s+1)}$ is holomorphic and we don't have any residue.

If $ \omega(a_M)= 1$, then $\mathfrak{S}_{a_M}(s)$ has a simple zero at $s=-1$ and thus $\mathfrak{S}_{a_M}(s)\zeta(s+1)\zeta(s+2) Z(s) \frac{M^s}{s(s+1)}$ has a simple pole with residue equal to 
$\frac{\phi(a_M)}{a_M}\frac{ \Lambda(a_M)}{2M}$.

Finally, if $a_M=1$, then $\mathfrak{S}_{a_M}(s)\equiv 1$ and thus $\mathfrak{S}_{a_M}(s)\zeta(s+1)\zeta(s+2) Z(s) \frac{M^s}{s(s+1)}$ has a double pole at $s=-1$ with residue equal to $\frac 12 \frac{\log M}M + \frac{C_5}{M} $.

\end{proof}

\section{Further results and proofs}
\label{section further results}

We will start by giving the fundamental result of this paper which works for $M$ fixed as well as for $M$ varying with $x$ under the condition $M\leq (\log x)^{O(1)}$.

\begin{proposition}
\label{M fixé}
Fix $A>B>0$ and $\lambda<\frac 14$, three positive real numbers. Let $M=M(x)$ be an integer such that $1\leq M(x)\leq(\log x)^B$.  For $a$ in the range $0<|a|\leq x^{\lambda}$ we have that
 \begin{multline}
\label{enonce M fixe}
\sum_{\substack{q\leq \frac x M  \\(q,a)=1}} \left( \psi(x;q,a)-\Lambda(a)-\frac{\psi(x)}{\phi(q)} \right) = x\Bigg({C}_1(a)\log M+{C}_3(a)-\sum_{\substack{r\leq M\\(r,a)=1}} \frac 1 {\phi(r)} \left( 1- \frac rM\right)\Bigg)\\+O_{A,B,\lambda}\left( \frac{\phi(a)}{a}2^{\omega(a)} \frac x {(\log x)^A} \right).
\end{multline}

We can remove the condition of $M$ being an integer at the cost of adding the error term $O(x\log \log M/M^2)$.

\end{proposition}

\begin{proof}
We will prove that
 \begin{multline}
\label{enonce M fixe avec x}
\sum_{\substack{q\leq \frac x M  \\(q,a)=1}} \left( \psi(x;q,a)-\Lambda(a)-\frac{x}{\phi(q)} \right) = x\Bigg({C}_1(a)\log M+{C}_3(a)-\sum_{\substack{r\leq M\\(r,a)=1}} \frac 1 {\phi(r)} \left( 1- \frac rM\right)\Bigg)\\+O_{A,B,\lambda}\left( \frac{\phi(a)}{a}2^{\omega(a)} \frac x {(\log x)^A} \right).
\end{multline}
From this we can deduce the proposition since by Lemma \ref{psi et x} the difference between the left hand side of \eqref{enonce M fixe} and that of \eqref{enonce M fixe avec x} is negligible. 

 Define $L:=(\log x)^{A+3}$. Partitioning the sum into dyadic intervals and applying Theorem \ref{FG} gives
\begin{equation}
\sum_{\substack{q\leq \frac xL \\(q,a)=1}} \left( \psi(x;q,a)- \Lambda(a)-\frac{x}{\phi(q)} \right) = O_{A,\lambda}\left( 2^{\omega(a)} \frac x {(\log x)^{A+1}} \right).
\label{partie petite de la somme}
\end{equation}
Therefore, we need to compute 
$$\sum_{\substack{\frac x{L} < q\leq \frac x M  \\(q,a)=1}} \left( \psi(x;q,a)-\Lambda(a) -\frac{x}{\phi(q)}\right),$$ 
which by lemmas \ref{psi et theta} and \ref{enlever *} is equal to
\begin{equation}
\label{somme qu'on va separer}
\sum_{\substack{\frac x{L} < q\leq \frac x M  \\(q,a)=1}} \left( \theta^*(x;q,a)-\frac{x}{\phi(q)}\right)+O(x^{2/3}+|a|(\log |a|)^2).
\end{equation} 
We split the sum in \eqref{somme qu'on va separer} in three distinct sums as following:
$$\sum_{\substack{\frac x{L} < q\leq x  \\(q,a)=1}} \theta^*(x;q,a)-\sum_{\substack{ \frac xM < q\leq x  \\(q,a)=1}} \theta^*(x;q,a) -x \sum_{\substack{\frac x{L} < q\leq \frac xM  \\(q,a)=1}} \frac  1 {\phi(q)}= I-II-III.$$
The third sum is easily treated using Lemma \ref{somme inverse euler}:
\begin{multline}
III=x\sum_{\substack{\frac x{L} < q\leq \frac xM  \\(q,a)=1}} \frac  1 {\phi(q)} = x\Bigg(   C_1(a) \log (x/M) +  C_1(a)+C_3(a) + O\left( 2^{\omega(a)}\frac{\log (x/M)}{x/M}\right) \\-  \left( C_1(a) \log (x/L) +  C_1(a)+C_3(a) + O\left( 2^{\omega(a)}\frac{\log (x/L)}{x/L} \right)\right) \Bigg)
\\ =  C_1(a) x\log(L/M) + O\left( 2^{\omega(a)}L \log x \right)  .
\label{somme euler 1}
\end{multline}
For the first sum, we have
\begin{equation*}
I= \sum_{\substack{\frac x{L} < q\leq x  \\(q,a)=1}} \theta^*(x;q,a) = \sum_{\substack{\frac x{L} < q\leq x  \\(q,a)=1}} \sum_{\substack{ |a|< p\leq x \\ p\equiv a \bmod q}} \log p .
\end{equation*}
Using Lemma \ref{interchanger les diviseurs},
\begin{align*}
 I &= \sum_{\substack{1\leq r < (x - a)\frac{L}{x}  \\(r,a)=1}} \sum_{\substack{ r\frac x{L} +a< p\leq x \\ p\equiv a \bmod r  \\ p> |a|}} \log p +O(|a|\log x)
\\&= \sum_{\substack{1\leq  r <(x - a)\frac{L}{x}  \\(r,a)=1}} \left(\theta^*(x;r,a)- \theta^*\left(r\frac x{L}+a;r,a\right) \right)+O(x^{1/3})
\\ &= \sum_{\substack{ 1\leq r < L - a \frac{L}{x} \\(r,a)=1}} \left(  \frac x {\phi(r)} - \frac {rx} { \phi(r) L}\right) + O(|a|L (\log x)^2)+O_A \left( L \frac x {L(\log x)^{A+1}}\right).
\end{align*}
In the last step we used Lemma \ref{enlever *} (with $|a|<x^{\lambda}$) combined with the Siegel-Walfisz theorem in the form  $\theta(x;r,a)-\frac x{\phi(r)} \ll_A \frac x{L(\log x)^{A+1}}$ for $r\leq 2 L$, as well as the estimate $\sum_{r\leq R} \frac {|a|} {\phi(r)} \ll |a|\log R$. As $\left| a \frac{L}{x}\right| \leq 1$ for $x$ large enough and $\phi(r)\gg r/\log\log r$, this gives
\begin{align*}
I &= x\sum_{\substack{ r <L  \\(r,a)=1}}  \frac 1 {\phi(r)}\left( 1- \frac{r}{L} \right) + O_A \left( \frac x {(\log x)^{A+1}}\right)+ O\left(  \frac {x\log\log L} {L}\right)
\\ &= x\sum_{\substack{ r <L  \\(r,a)=1}}  \frac 1 {\phi(r)}\left( 1- \frac{r}{L} \right) + O_A \left( \frac x {(\log x)^{A+1}}\right).
\end{align*}
We conclude the evaluation of $I$ by applying Lemma \ref{somme inverse euler}:
\begin{equation}
I= x\left( C_1(a) \log L +  C_3(a) + O_A\left(  \frac{2^{\omega(a)}}{(\log x)^{A+1}}\right) \right).
\end{equation}
Now with a similar computation using lemmas \ref{interchanger les diviseurs} and \ref{enlever *} as well as the Siegel-Walfisz theorem in the form  $\theta(x;r,a)-\frac x{\phi(r)} \ll_{A,B} \frac x{M (\log x)^{A+1}}$ for $r\leq 2M$, we show that 
\begin{align*}
 II &= \sum_{\substack{1\leq  r < (x - a)\frac{M}{x}  \\(r,a)=1}} \left(\theta^*(x;r,a)- \theta^*\left(r\frac x{M}+a,r,a\right) \right)+O(|a|\log x)
\\ &= \sum_{\substack{ 1\leq r < M - a \frac{M}{x} \\(r,a)=1}} \left(  \frac x{\phi(r)} - \frac {rx} { \phi(r)M}\right) + O(|a| M (\log x)^2)+O_{A,B}\left( M \frac x {M (\log x)^{A+1}}\right)
\\ &=  x\sum_{\substack{ 1\leq r \leq M \\(r,a)=1}} \frac 1{\phi(r)} \left(  1 - \frac rM \right) +O_{A,B} \left(  \frac x {(\log x)^{A+1}}\right).
\end{align*}
In the last step we have used that $M$ is an integer so the term for $r=M$ in the sum is given by $\frac 1 {\phi(M)} \left( 1-\frac MM\right) = 0$. If $M\notin \mathbb{N}$, we have to add the error term $ \frac x{\phi(\lfloor M\rfloor)} \left(  1 - \frac {\lfloor M\rfloor}M \right)= O(x \log\log M /M^2)$.

We conclude the proof by combining our estimates for $I$, $II$ and $III$ with the bound $\frac{a}{\phi(a)} \ll \log\log x$.
\end{proof}

\begin{proof}[Proof of Proposition \ref{M=1}]
 Follows from Proposition \ref{M fixé}.
\end{proof}

For $M$ not necessarily fixed, we will need to use Lemma \ref{somme inverse euler avec poids}. 

\begin{theorem}
\label{theoreme technique}
Let $M=M(x) \leq (\log x)^B$ (not necessarily an integer) where $B>0$ is a fixed real number. Fix $\lambda<\frac 14$ a positive real number, and let $a\neq 0$ be an integer such that $|a|\leq x^{\lambda}$.

If $\omega(a_M)\geq 1$,
\begin{equation}
\frac1{\frac{\phi(a)}{a}\frac{x}{M}} \sum_{\substack{q\leq \frac x {M}  \\(q,a)=1}} \left( \psi(x;q,a)-\Lambda(a)-\frac{\psi(x)}{\phi(q)} \right) = - \frac 12 \Lambda(a_R)+ E(M,a).
\end{equation}

If $a_M=1$,
\begin{equation}
\frac1{\frac{\phi(a)}{a}\frac{x}{M}}\sum_{\substack{q\leq \frac x M  \\(q,a)=1}} \left( \psi(x;q,a)-\Lambda(a)-\frac{\psi(x)}{\phi(q)} \right) = - \frac 12 \log M-C_5 +E(M,a).
\end{equation}

Under the assumption that $\sum_{\substack{p\mid a \\ p>M}} \frac{1} p \leq 1$, the error term $E(M,a)$ satisfies, for some $\delta>0$, 
\begin{multline}
\label{borne pour erreur dans thm technique}
E(M,a) \ll_{a,\epsilon,B}\frac{a}{\phi(a)}\prod_{p\mid a_M}\left( 1+\frac 1{p^{\delta}}\right) \left(\frac{a_M}{M}\right)^{\frac{205}{538}-\epsilon}+M\log M \sum_{\substack{p\mid a \\ p>M}} \frac{\log p} p  \\
+ \frac{2^{\omega(a)}}{(\log x)^{B+1}} + \frac a{\phi(a)} \frac{\log\log M}{M}.
\end{multline}
\end{theorem}


\begin{remark}
If we assume Lindelöf's hypothesis, we can replace the first term of the right hand side of \eqref{borne pour erreur dans thm technique} by $\frac a{\phi(a)}\prod_{p\mid a_M}\left( 1+\frac 1{p^{\delta}}\right) \left(\frac{a_M}{M}\right)^{1/2-\epsilon}$.
\end{remark}

\begin{proof}[Proof of Theorem \ref{theoreme technique}]
 Our starting point will be to set $A:=2B+1$ in Proposition \ref{M fixé}:
\begin{multline}
\label{somme qu'on va remplacer par R}
\sum_{\substack{q\leq \frac x M  \\(q,a)=1}} \left( \psi(x;q,a)-\Lambda(a)-\frac{\psi(x)}{\phi(q)}  \right) = x\Bigg({C}_1(a)\log M+{C}_3(a)-\sum_{\substack{r\leq M\\(r,a)=1}} \frac 1 {\phi(r)} \left( 1-\frac rM\right)\Bigg)\\+O_{B}\left( \frac{\phi(a)}{a} \frac xM \frac{2^{\omega(a)}}{(\log x)^{B+1}}+x\frac{\log \log M}{M^2} \right).
\end{multline}


We now consider three cases depending on the number of prime factors of $a_M$.

Case 1: $\omega(a_M)\geq 2$, which implies $\omega(a)\geq 2$. Applying Lemma \ref{somme inverse euler avec poids} gives 
\begin{multline}
\sum_{\substack{q\leq \frac x M  \\(q,a)=1}} \left( \psi(x;q,a)-\Lambda(a)-\frac{\psi(x)}{\phi(q)} \right) \ll_B x\Big( | C_1(a_M)- C_1(a)| \log M +  | C_3(a_M)- C_3(a)| \\+\frac{\prod_{p\mid a_M}\left( 1+\frac 1{p^{\delta}}\right)} M \left(\frac{a_M} M \right)^{\frac{205}{538}-\epsilon} \Big) + \frac{\phi(a)}{a} \frac xM \frac{2^{\omega(a)}}{(\log x)^{B+1}}+x\frac{\log \log M}{M^2}.
\end{multline}
If all the prime factors of $a$ are less than or equal to $M$, then $| C_i(a_M)- C_i(a)|=0$. If not, we need upper bounds. By the definition of $C_1(a)$, 
\begin{align*}
| C_1(a_M)- C_1(a)| &= C_1(a)\Bigg(\prod_{\substack{p\mid a \\ p> M}} \left(1-\frac 1p\right)^{-1} \left(1-\frac 1{p^2-p+1}\right)^{-1}-1\Bigg) \\
&\ll \frac{\phi(a)}{a} \sum_{\substack{p\mid a \\ p>M}} \frac 1p,
\end{align*}
as long as $\sum_{\substack{p\mid a \\ p>M}} \frac 1p\leq 1$. Moreover,
\begin{align*}
 C_3(a_M)-C_3(a)&= \Bigg(C_1(a) +O\Big(  \frac{\phi(a)}{a} \sum_{\substack{p\mid a \\ p>M}} \frac 1p \Big)\Bigg) \frac{C_3(a_M)}{C_1(a_M)} -C_3(a)
\\  &= C_1(a) \left( \frac{C_3(a_M)}{C_1(a_M)} -\frac{C_3(a)}{C_1(a)}\right) + O\Bigg(  \frac{\phi(a)}{a} \sum_{\substack{p\mid a \\ p>M}} \frac{\log M}p  \Bigg),
\end{align*}
since $\sum_{p\mid a_M} \frac{p^2 \log p}{(p-1)(p^2-p+1)} \ll \sum_{p\leq M} \frac{\log p}p \ll \log M$. Thus,
\begin{align*}
 |C_3(a_M)-C_3(a)| &\ll \frac{\phi(a)}{a}\sum_{\substack{p\mid a \\ p> M}} \frac{p^2 \log p}{(p-1)(p^2-p+1)}+\frac{\phi(a)}{a} \sum_{\substack{p\mid a \\ p>M}} \frac{\log M}p  \\  &\ll \frac{\phi(a)}{a}  \sum_{\substack{p\mid a \\ p>M}} \frac{\log p}p .
\end{align*}
Putting all this together and dividing by $\frac xM\frac{\phi(a)}{a}$ gives the claimed estimate.

Case 2: $\omega(a_M)=1$. The calculation is similar, however Lemma \ref{somme inverse euler avec poids} gives a contribution of $$-\frac 12 \frac{a}{\phi(a)} \frac{\phi(a_M)}{a_M}\Lambda(a_M)=-\frac 12\Lambda(a_M) + O\Bigg( \log M \sum_{\substack{p \mid a \\ p>M}} \frac 1p \Bigg),$$
since $\Lambda(a_M) \leq \log M$.

Case 3: $a_M=1$. The contribution of $$\frac a{\phi(a)}\left(-\frac 12 \log M -C_5\right)=-\frac 12 \log M -C_5+O\Bigg( \log M\sum_{\substack{p \mid a \\ p>M}} \frac 1p \Bigg)$$ comes from Lemma \ref{somme inverse euler avec poids}.

\end{proof}

\begin{remark}
One can give an complex analytic explanation of why the main terms all cancel out in the proofs of Proposition \ref{M fixé} and Theorem \ref{theoreme technique}. Indeed, it amounts to the fact that the Mellin transform 
$$\psi(s):=\frac {L^s-M^s}{s+1}
+\left(\frac{x}{L}\right)^s-\left(\frac{x}{M}\right)^s $$
has a double zero at $s=0$ (see the proof of Proposition 5.1 of \cite{fiorilli}).
\end{remark}

\begin{proof}[Proof of Theorem \ref{premier théorème}]

It is a particular case of Theorem \ref{theoreme technique}. The theorem is trivial if $M$ is bounded. For $M$ tending to infinity with $x$, we have that $a_M=a$ for $x$ large enough, since $a$ is fixed. Hence, the error term in Theorem \ref{theoreme technique} satisfies $E(M,a)\ll_{a,\epsilon,B} M^{-\frac{205}{538}+\epsilon}$.
\end{proof}

\section{Concluding remarks}

\begin{remark}
 One could ask if the results of Theorem \ref{premier théorème} are intrinsic to the sequence of prime numbers or if they are just a result of the weight $\Lambda(n)$ in the prime counting functions. However one can see that if we replace $\psi(x;q,a)$ by $\pi(x;q,a)$ and $\psi(x)$ by $\pi(x)$, the proof of Proposition \ref{M fixé} will go through with $x$ replaced by $Li(x)$ and an additional error term of $$O\left( x\frac{\log\log x}{(\log x)^2}\right).$$
(This is because in the evaluation of $II$ in the proof of Proposition \ref{M fixé}, we have for $\frac 1{\log x} \leq y\leq 1+o(1)$ that $Li(yx)-y Li(x) \ll \frac{|\log y|}{(\log x)^2}$, a bound which cannot be improved for small values of $y$.)
 One has to prove the analogue of Theorem \ref{FG} which can be done using the Bombieri-Vinogradov theorem and a very delicate summation by parts. We conclude that an analogue of Theorem \ref{premier théorème} holds with the natural prime counting functions in the range $M= o\left(\frac{\log x}{\log\log x}\right)$.
\end{remark}

\begin{theorem}
\label{premier théorème avec pi}
Fix an integer $a\neq 0$ and $\epsilon>0$. If $M=M(x)$ is such that $M= o\left(\frac{\log x}{\log\log x}\right)$, then

\begin{equation*}
\frac1{\frac{\phi(a)}{a}\frac{Li(x)}{M} }\sum_{\substack{q\leq \frac{x} {M}  \\(q,a)=1}} \left( \pi(x;q,a)- \frac{\vartheta(a)}{\log a}-\frac{\pi(x)}{\phi(q) } \right) = \mu(a,M)+O_{a,\epsilon} \left(\frac{1}{M^{\frac{205}{538}-\epsilon}}+\frac{M \log\log x}{\log x}\right),
\end{equation*}
where $\mu(a,M)$ is defined as in Theorem \ref{premier théorème}.
\end{theorem}

\begin{remark}

Using Theorem \ref{theoreme technique}, one can show that Theorem \ref{premier théorème} holds for many (not necessarily fixed) values of $a$:
\begin{theorem}
\label{théorème proportion positive}
Fix $B$, $\lambda< \frac 14$ and $\eta\leq \frac 15$, three positive real numbers. We have for $M\leq (\log x)^B$ that the proportion of integers $a$ in the range $0<|a|\leq x^{\lambda}$ for which 
\begin{equation}
\label{equation du thm prop positive}
\frac1{\frac{\phi(a)}{a}\frac{x}{M}  }\sum_{\substack{q\leq \frac x M  \\(q,a)=1}} \left( \psi(x;q,a)-\frac{\psi(x)}{\phi(q)} \right) = O_{B,\lambda}\left(\frac  1 {M^{\eta}} \right)
\end{equation}
is at least $\frac{1-\frac{21}{8}\eta}{1+\eta+\eta^2}e^{-\gamma}$,
where $\gamma$ is the Euler-Mascheroni constant.
\end{theorem}

We now sketch an argument showing that the proportion of integers $a\leq x^{\lambda}$ for which the first term of the right hand side of \eqref{borne pour erreur dans thm technique} is $\leq 1$ is not more than $e^{-\gamma}$. Setting $X:=x^{\lambda}$ and using that $M\ll (\log X)^{O(1)}$, we compute

\begin{equation*}
 \#\Big\{ a\leq X : \prod_{\substack{p\mid a \\ p\leq M}} p \leq M \Big\} = \sum_{\substack{n\leq M \\ \mu^2(n)=1}} \# \Big\{ a \leq X : n\mid a, \Big( \frac an,\prod_{\substack{p\leq M \\ p \nmid n}}p \Big)=1 \Big\},
\end{equation*}
which by the fundamental lemma of the combinatorial sieve is

\begin{equation*}
\sim \sum_{\substack{n\leq M \\ \mu^2(n)=1}} \frac Xn \prod_{\substack{p\leq M \\ p \nmid n}} \left( 1-\frac 1p \right)  \sim \frac{X e^{-\gamma}}{\log M} \sum_{n\leq M} \frac{\mu^2(n)}{\phi(n)} \sim X e^{-\gamma}.
\end{equation*}
Therefore, to extend the proportion of integers $a$ for which we get an admissible error term in Theorem \ref{théorème proportion positive}, one would need to improve the bound \eqref{borne pour erreur dans thm technique} for $E(M,a)$.

It would be interesting to either extend this proportion to $1$, or to show that this is impossible, that is to show that that the left hand side of \eqref{equation du thm prop positive} is $\Omega(1)$ for a positive proportion of $a$ in the range $0<|a|\leq x^{\lambda}$.
\end{remark}

\begin{remark}
If we consider dyadic sums, that is we sum over the moduli $\frac x{2M}<q\leq \frac xM$ with $(q,a)=1$ as in \eqref{enonce premier thm dyadique}, then we can improve the range of $a$ in our results to $|a|\leq \frac x{(\log x)^C}$ for a large enough constant $C$, since in this case we do not need to appeal to Theorem \ref{FG}. (It was already remarked by the authors in \cite{FG} that if $Q$ is not in the interval $[x^{\frac 12-\epsilon},x^{\frac 12+\epsilon}]$, then Theorem \ref{FG} holds in a much wider range of $a$.)
\end{remark}

\begin{remark}
If we assume GRH and we use dyadic intervals, then we can improve the range of $M$ in Theorem \ref{premier théorème}. Indeed we can show that for $M=M(x)=o(x^{\frac 14}/\log x)$ and for any $\epsilon>0$,
\begin{multline}
\label{enonce premier thm dyadique}
\frac1{\frac{\phi(a)}{a}\frac{x}{2M}  }\sum_{\substack{\frac x{2M}<q\leq \frac{x} {M}  \\(q,a)=1}} \left( \psi(x;q,a)-\Lambda(a)-\frac{\psi(x)}{\phi(q)}\right) = 2\mu(a,M)-\mu(a,2M)\\
+O_{a,\epsilon} \left(\frac{M^2}{\sqrt x} (\log x)^2+\frac{1}{M^{\frac {1}{2}-\epsilon}}\right).
\end{multline}
If we assume Montgomery's hypothesis, then we can show that the error term in \eqref{enonce premier thm dyadique} is $o(1)$ for $M\leq x^{\frac 13-\epsilon}$, for any fixed $\epsilon>0$.
\end{remark}

\begin{remark}
What the proof of Theorem \ref{premier théorème} boils down to is the distribution of primes in arithmetic progressions, rather than depending on the primes themselves. For this reason, one can generalize Theorem \ref{premier théorème} to many arithmetic sequences that are well distributed in arithmetic progressions. This will be done in a forthcoming article \cite{fiorilli}.

\end{remark}

\end{document}